\documentclass[a4paper,12pt]{amsart}
\usepackage[colorlinks,linkcolor=blue,citecolor=blue]{hyperref}
\usepackage{latexsym, amssymb, amsmath, amscd, amsthm, mathrsfs, bbm}
\usepackage[all, knot]{xy}
\xyoption{arc}

\usepackage{anysize}\marginsize{22mm}{22mm}{25mm}{25mm}
\allowdisplaybreaks[4]

\def \k {\mathbbm{k}}

\def \dim {\operatorname{dim}}

\def \C {\mathbbm{C}}

\numberwithin{equation}{section}
\numberwithin{table}{section}
\numberwithin{equation}{section}
\newtheorem{theorem}{Theorem}[section]

\newtheorem{proposition}[theorem]{Proposition}

\newtheorem{example}[theorem]{Example}

\newtheorem{remark}[theorem]{Remark}

\begin{document}

\title{Solving cubic equations by completing the cube and higher degree equations by completing powers}
\thanks{Supported by NSFC 11971181 and 11971449.}



\author{Hua-Lin Huang, Shengyuan Ruan, Xiaodan Xu and Yu Ye}

\address{School of Mathematical Sciences, Huaqiao University, Quanzhou 362021, China}
\email{hualin.huang@hqu.edu.cn}

\address{School of Mathematical Sciences, Huaqiao University, Quanzhou 362021, China}
\email{shengyuan.ruan@stu.hqu.edu.cn}

\address{School of Mathematical Sciences, Huaqiao University, Quanzhou 362021, China}
\email{xiaodan.xu@stu.hqu.edu.cn}

\address{School of Mathematical Sciences, Wu Wen-Tsun Key Laboratory of Mathematics, University of Science and Technology of China, Hefei 230026, China}
\email{yeyu@ustc.edu.cn}

\date{}                                           

\maketitle


\begin{abstract}
We derive the Cardano formula of cubic equations by completing the cube, and provide radical solutions to some algebraic equations of higher degree by completing powers. The main idea of completing powers arises from Harrison's center theory of higher degree forms. A very simple criterion for such algebraic equations is presented, and the computation amounts to solving linear equations and quadratic equations.
\end{abstract}

\vskip 20pt

\section{Introduction}

Solving algebraic equations is a key problem throughout the whole history of mathematics. The radical solution to a quadratic equation was found by the Babylonians 1500 BC. Three thousands years later, cubic and quartic equations were solved in terms of radicals by Italian mathematicians. After the works of Ruffini, Abel and Galois, it is a common knowledge that a general algebraic equation of degree at least $5$ has no radical solution. See \cite{vdw} for a history. Theoretically, an equation is solvable by radical if and only if its Galois group is solvable. In practice, however, it is very hard to determine whether an equations is solvable by radical.

This short article first comes out of the frequently asked question: can one solve cubic equations by completing the cube? There have been a lot of attempts in the literature. For example, Kung and Rota found in \cite{kr} radical solutions to cubic equations by the classical invariant theory of binary forms, where the completion of cubes is by virtue of canonical forms of binary cubics. Recently, in \cite{w} Wallach found a linear fractional transformation that completes the cube for a cubic by geometric invariant theory. The aforementioned works involve complicated computations. The key approach of  the present article is an elementary method of completing powers for higher degree forms based on Harrison's theory of centers \cite{h1, hllyz, hlyz, hlyz2}. Certainly this can be applied to find radical solutions of some algebraic equations. 

Let $f(x)=a_0x^d+{d \choose 1}a_1x^{d-1}+\cdots+{d \choose d-1}a_{d-1}x+a_d=0$ be a complex algebraic equation of degree $d>2$ and let $F(x,y)=a_0x^d+{d \choose 1}a_1x^{d-1}y+\cdots+{d \choose d-1}a_{d-1}xy^{d-1}+a_dy^d$ be its homogenization. Let $H$ be the hessian matrix of $F(x,y).$ The center $Z(F):=\{ X \in \C^{2 \times 2} \mid (HX)^T=HX \}$ is a subalgebra of the full matrix algebra $\C^{2 \times 2}.$ It turns out that $Z(F) \cong \C \times \C$ if and only if $F(x,y)=(\alpha_1 x + \beta_1 y)^d + (\alpha_2 x + \beta_2 y)^d,$ and if and only if $f(x)=(\alpha_1 x + \beta_1)^d + (\alpha_2 x + \beta_2)^d,$ for some $\alpha_i$ and $\beta_j$ with $\alpha_1\beta_2-\alpha_2\beta_1 \ne 0.$ If this is the case, then $f(x)=0$ is solvable by radical and the roots are easily obtained. Moreover, the completion of powers is fully indicated by the center and the computation is elementary involving only solving simple linear equations and quadratic equations. 

The article is organized as follows. In Section 2 we recall Harrison's theory of centers and its application to completing powers of polynomials. Then we apply the center theory to derive the Cardano formula of cubic equations in Section 3, and to solve some higher degree algebraic equations by completing powers in Section 4.

\section{Harrison centers and completing powers}

Throughout this section, let $d>2$ be an integer and $\k$ be a field of characteristic $0$ or $>d.$ As preparation, we recall the notion of center algebras of homogeneous polynomials and its application to some criterion and algorithm of completing powers.

Let $f(x_1, x_2, \dots, x_n) \in \k[x_1, \dots, x_n]$ be a homogeneous polynomial of degree $d.$ The center of $f=f(x_1, x_2, \dots, x_n)$ was introduced by Harrison \cite{h1} as follows:
$$Z(f):=\{ X \in \k^{n \times n} \mid (HX)^T=HX \},$$ where $H=(\frac{\partial^2 f}{\partial x_i \partial x_j})_{1 \le i,\ j \le n}$ is the hessian matrix.

\begin{example} \label{sp}
\emph{Consider the sum of powers $f=x_1^d+x_2^d+\cdots+x_n^d.$ Then the hessian matrix of $f$ is diagonal with $ii$-entry $d(d-1)x_i^{d-2}.$ Thus the $ij$-entry of $HX$ reads $d(d-1)x_i^{d-2}c_{ij}$ for $X=(c_{ij}).$ According to the symmetric condition of $HX,$ it is direct to see that $Z(f)$ consists of all diagonal matrices. Thus $Z(f)$ is isomorphic to $\k^n$ as algebras.}
\end{example}

We call a homogeneous polynomial nondegenerate, if no variable can be removed by an invertible linear change of its variables. It is well known that a homogeneous polynomial $f$ is nondegenerate if and only if its first-order differentials $\frac{\partial f}{\partial x_i}$ are linearly independent \cite{h1}. Moreover, any homogeneous polynomial is the sum of a nondegenerate polynomial and a zero polynomial. We say that $f(x_1, x_2, \dots, x_n)$ is diagonalizable over $\k$ if there exists an invertible $\k$-linear change of variables $x=Py$ such that 
\begin{equation*} \label{df}
f(x)=f(Py)=\lambda_1y_1^d+\lambda_2y_2^d+\cdots+\lambda_ry_r^d,
\end{equation*}
where $\lambda_i \in \k.$ Note that $r$ is necessarily less than or equal to $n$ here.

The center turns out to be a very effective invariant for deciding whether a homogeneous polynomial is diagonalizable. The following proposition was essentially obtained in \cite{h1}, see also \cite{hllyz, hlyz, hlyz2}.

\begin{proposition} \label{cd}
Suppose $f \in \k[x_1, x_2, \dots, x_n]$ is a nondegenerate homogeneous polynomial of degree $d>2.$ Then $Z(f)$ is a commutative algebra, and $f$ is diagonalizable over $\k$ if and only if $Z(f) \cong \k^n$ as algebras.
\end{proposition}

\begin{proof}
For the convenience of the reader, we include a sketchy proof for the last assertion. If $f$ is nondegenerate and diagonalizable, then there is a change of variable $x=Py$ such that 
\[ g(y):=f(Py)=\lambda_1y_1^d+\lambda_2y_2^d+\cdots+\lambda_ny_n^d \] with $\lambda_i \in \k$ and $\prod_{i=1}^n \lambda_i \ne 0.$ Let $G$ denote the hessian matrix $(\frac{\partial^2g}{\partial y_i \partial y_j})$ of $g$ with respect to the variables $y_i.$ Then by the chain rule of derivative, it is easy to see that $G=P^THP,$ where $H=(\frac{\partial^2 f}{\partial x_i \partial x_j}).$ Therefore, \[(GY)^T=GY \Leftrightarrow (HPYP^{-1})^T=HPYP^{-1}.\] It follows that $Z(f) = PZ(g)P^{-1}.$ As Example \ref{sp}, it is ready to see that $Z(g)$ consists of all diagonal matrices. Hence $Z(f)$ is isomorphic to $\k^n$ as algebras.

Conversely, suppose $Z(f) \cong \k^n$ as algebras. Then in the center algebra $Z(f)$ there exists a complete set of orthogonal primitive idempotents, say $e_1, e_2, \dots, e_n.$ As the $e_i$ are diagonalizable under conjugation and they mutually commute, there exists an invertible matrix $P \in \k^{n \times n}$ such that $P^{-1}e_iP=E_{ii}$ for all $1 \le i \le n,$ where $E_{ii}$ is the canonical idempotent matrix with $ii$-entry $1$ and $0$ otherwise. Now take the change of variables $x=Py,$ and denote $g(y)=f(Py).$ Note that $Z(g)=P^{-1}Z(f)P,$ it follows that $E_{ii} \in Z(g).$ By the definition of center algebras, it is direct to see that the hessian matrix of $g$ is diagonal. It follows easilly that $g(y)=\lambda_1y_1^d+\lambda_2y_2^d+\cdots+\lambda_ny_n^d,$ i.e. $f(x)$ is diagonalizable.
\end{proof}

\begin{remark}
\emph{The previous proposition provides a criterion for completing powers of multivariate homogeneous polynomials. In addition, the proof also contains an algorithm for the process of completing powers for diagonalizable polynomials.}
\end{remark}

We conclude this section with an example to elucidate the algorithm of completing powers.

\begin{example}
\emph{Consider the following ternary cubic
\begin{eqnarray*} f(x_1, x_2, x_3)&=&
x_1^3 + 3 x_2 x_1^2 + 3 x_3 x_1^2 + 3 x_2^2 x_1 + 3 x_3^2 x_1 \\ &+& 6 x_2 x_3 x_1 - x_2^3 + 20 x_3^3 - 21 x_2 x_3^2 + 15 x_2^2 x_3.
\end{eqnarray*}
Suppose $X=(c_{ij}) \in Z(f).$ The condition $(HX)^T=HX$ can be translated into a system of linear equations in $x_{ij}.$ A general solution to the system of linear equations reads $$X=a\begin{pmatrix} 1 & 1 & 1\\ 0 & 0 & 0\\ 0 & 0 & 0 \end{pmatrix}+b\begin{pmatrix} 0 & -1 & 2\\ 0 & 1 & -2\\ 0 & 0 & 0 \end{pmatrix}+ c\begin{pmatrix} 0 & 0 & -3\\ 0 & 0 & 2\\ 0 & 0 & 1 \end{pmatrix}.$$ Let $P=\begin{pmatrix} 1 & -1 & -3\\ 0 & 1 & 2\\ 0 & 0 & 1 \end{pmatrix}.$ By direct verification, we have
\[ P^{-1}\begin{pmatrix} 1 & 1 & 1\\ 0 & 0 & 0\\ 0 & 0 & 0 \end{pmatrix}P=E_{11}, \ \ P^{-1}\begin{pmatrix} 0 & -1 & 2\\ 0 & 1 & -2\\ 0 & 0 & 0 \end{pmatrix}P=E_{22}, \ \ P^{-1}\begin{pmatrix} 0 & 0 & -3\\ 0 & 0 & 2\\ 0 & 0 & 1 \end{pmatrix}P=E_{33}. \] Now take a change of variables $x=Py,$ and by direct computation we have 
\[ f(Py)=y_1^3-2y_2^3+3y_3^3.\] Note that $y_1=x_1 + x_2 + x_3, \ y_2=x_2 - 2 x_3, \ y_3=x_3,$ therefore 
\[ f(x_1, x_2, x_3)=(x_1 + x_2 + x_3)^3 - 2 (x_2 - 2 x_3)^3 + 3 x_3^3. \]}
\end{example}

\section{Cardano formula revisited by completing the cube}
In this section, we derive the well-known Cardano formula of cubic equations by completing the cube. Let \begin{equation} \label{cubic} a_0x^3+3a_1x^2+3a_2x+a_3=0 \end{equation} be a general cubic equation over the field $\C$ of complex numbers. Similar to the situation of quadratic equations but going a bit further, we aim to express cubic equations as the sum of two cubes of linear binomials.  That is, we try to find an identity as the following \begin{equation} \label{sc} f(x)=a_0x^3+3a_1x^2+3a_2x+a_3=(\alpha_1 x +\beta_1)^3 + (\alpha_2 x +\beta_2)^3. \end{equation} If such an identity is found, then the solution to \eqref{cubic} is easily obtained by taking the cubic roots of the both sides of $(\alpha_1 x +\beta_1)^3 =-(\alpha_2 x +\beta_2)^3.$

Thanks to the theory of centers, we have effective criterion and algorithm of completing the cubes for a cubic equation as \eqref{sc}. In order to apply Proposition \ref{cd}, we homogenize cubic equations as binary cubics \begin{equation} \label{bc} F(x,y)=a_0x^3+3a_1x^2y+3a_2xy^2+a_3y^3. \end{equation} It is clear that \begin{equation} \label{deh} f(x)=(\alpha_1 x +\beta_1)^3 + (\alpha_2 x +\beta_2)^3 \Leftrightarrow F(x,y)=(\alpha_1 x +\beta_1y)^3 + (\alpha_2 x +\beta_2y)^3. \end{equation} 

We start with computing the center of $F(x,y).$ Suppose $X=(c_{ij}) \in \C^{2 \times 2}$ and $X \in Z(F).$ Then the $c_{ij}$ satisfy the following linear equations
\begin{equation} \label{lec}
\begin{cases}
a_0c_{12}+a_1(c_{22}-c_{11})-a_2c_{21}=0 \\
a_1c_{12}+a_2(c_{22}-c_{11})-a_3c_{21}=0
\end{cases}
\end{equation}

First we derive the following well-known facts by the theory of centers.

\begin{proposition} \label{onecube}
Let $F(x,y)$ be as \eqref{bc}.
\begin{itemize}
\item[(1)] $F(x,y)$ is a perfect cube if and only if $\frac{a_0}{a_1} = \frac{a_1}{a_2} = \frac{a_2}{a_3}.$
\item[(2)] $F(x,y)=(\alpha x +\beta y)^3+\gamma y^3$ if and only $\frac{a_0}{a_1} = \frac{a_1}{a_2} .$
\item[(3)] $F(x,y)=\gamma x^3+(\alpha x +\beta y)^3$ if and only $\frac{a_1}{a_2} = \frac{a_2}{a_3} .$
\end{itemize}
\end{proposition}

\begin{proof}
(1) Suppose $F(x,y)$ is a perfect cube. That is, $F(x,y)$ is degenerate. Then the center $Z(F)$ is not commutative and in this case $\dim Z(F)=3.$ In other words, by \eqref{lec} the matrix \[ \begin{pmatrix} a_0 & a_1 & a_2 \\ a_1 & a_2 & a_3 \end{pmatrix} \] is of rank $1,$ and thus $\frac{a_0}{a_1} = \frac{a_1}{a_2} = \frac{a_2}{a_3}.$ On the contrary, suppose $\frac{a_0}{a_1} = \frac{a_1}{a_2} = \frac{a_2}{a_3}.$ Then the previous matrix is of rank $1,$ so $Z(F)$ is $3$-dimensional. As is clear that commutative subalgebras of $\C^{2 \times 2}$ are at most $2$-dimensional, that implies $F(x,y)$ is degenerate in this case. Therefore, $F(x,y)=(\alpha x + \beta y)^3$ for some $\alpha, \beta \in \C.$ 

(2) Suppose $F(x,y)=(\alpha x +\beta y)^3+\gamma y^3.$ Since $a_0, a_1, a_2$ are determined by the perfect cube $(\alpha x +\beta y)^3,$ by (1) we have $\frac{a_0}{a_1} = \frac{a_1}{a_2} .$ Cnversely, suppose $\frac{a_0}{a_1} = \frac{a_1}{a_2} .$ We take an $a \in \C$ such that $\frac{a_0}{a_1} = \frac{a_1}{a_2} = \frac{a_2}{a}.$ Then by (1) we have \begin{eqnarray*} F(x,y)&=&a_0x^3+3a_1x^2y+3a_2xy^2+a_3y^3 \\ &=& a_0x^3+3a_1x^2y+3a_2xy^2+ay^3+a_3y^3-ay^3 \\ &=& (\alpha x +\beta y)^3+\gamma y^3. \end{eqnarray*}

(3) Similar to (2).
\end{proof}

It follows from the previous proposition that only very special cubic equations can be completed as one perfect cube plus a constant, in contrast to the situation of quadratic equations. Instead, we show in the following that a general cubic equation can be completed as the sum of two perfect cubes of linear binomials as \eqref{deh}. 

\begin{theorem}
Let $F(x,y)=a_0x^3+3a_1x^2y+3a_2xy^2+a_3y^3.$ Then $$F(x,y)=(\alpha_1 x +\beta_1 y)^3 + (\alpha_2 x +\beta_2 y)^3$$ with $\begin{vmatrix} \alpha_1 & \beta_1 \\ \alpha_2 & \beta_2 \end{vmatrix} \ne 0$ if and only if $$\begin{vmatrix} a_0 & a_2 \\  a_1 & a_3 \end{vmatrix}^2 -4\begin{vmatrix} a_0 & a_1 \\  a_1 & a_2 \end{vmatrix}\begin{vmatrix} a_1 & a_2 \\  a_2 & a_3 \end{vmatrix}   \ne 0.$$ Consequently, a general complex binary cubic is the sum of cubes of two different linear forms. \end{theorem}

\begin{proof}
According to Proposition \ref{cd}, $F(x,y)=(\alpha_1 x +\beta_1 y)^3 + (\alpha_2 x +\beta_2 y)^3$ with $\begin{vmatrix} \alpha_1 & \beta_1 \\ \alpha_2 & \beta_2 \end{vmatrix} \ne 0$ if and only if $Z(F) \cong \C \times \C.$ On the other hand, under the conditions of the theorem the coefficient matrix of the system of linear equations \eqref{lec} has rank $2$, so by direct compution a general element of $Z(F)$ reads
\[ \begin{pmatrix} c_{11} & c_{12} \\ c_{21} & c_{22} \end{pmatrix} = a \begin{pmatrix} 1 & 0 \\ 0 & 1\end{pmatrix} + b \begin{pmatrix} 0 & -\begin{vmatrix} a_1 & a_2 \\  a_2 & a_3 \end{vmatrix} \\ \begin{vmatrix} a_0 & a_1 \\  a_1 & a_2 \end{vmatrix} & \begin{vmatrix} a_0 & a_2 \\  a_1 & a_3 \end{vmatrix} \end{pmatrix}, \quad \forall a, b \in \C. \] 
Let $\Lambda$ denote the last matrix and $T(\lambda)$ denote its characteristic polynomial \[ \lambda^2 - \begin{vmatrix} a_0 & a_2 \\  a_1 & a_3 \end{vmatrix} \lambda +  \begin{vmatrix} a_0 & a_1 \\  a_1 & a_2 \end{vmatrix} \begin{vmatrix} a_1 & a_2 \\  a_2 & a_3 \end{vmatrix}. \] It is clear that $Z(F) \cong \C[\lambda]/(T(\lambda)).$ Therefore, by the Chinese Remainder Theorem, $Z(F) \cong \C \times \C$ if and only if $T(\lambda)$ has no multiple roots, if and only if $\begin{vmatrix} a_0 & a_2 \\  a_1 & a_3 \end{vmatrix}^2 -4 \begin{vmatrix} a_0 & a_1 \\  a_1 & a_2 \end{vmatrix}\begin{vmatrix} a_1 & a_2 \\  a_2 & a_3 \end{vmatrix} \ne 0.$ 

The preceding assertion implies that the set of binary cubics that are sums of cubes of two different linear forms is a principle open set in the affine space of all binary cubics. As is obvious that this set is not empty, hence it is dense by the theory of elementary algebraic geometry \cite[Chapter 4]{clo}. In other words, a general binary cubic is the sum of cubes of two different linear forms.
\end{proof}

Now we complete the cube for a general binary cubic. Keep the notation 
\begin{equation} \label{mL}
\Lambda =\begin{pmatrix} 0 & -\begin{vmatrix} a_1 & a_2 \\  a_2 & a_3 \end{vmatrix} \\ \begin{vmatrix} a_0 & a_1 \\  a_1 & a_2 \end{vmatrix} & \begin{vmatrix} a_0 & a_2 \\  a_1 & a_3 \end{vmatrix} \end{pmatrix}.
\end{equation}
From now on, for brevity we denote \[ D_1=\begin{vmatrix} a_0 & a_1 \\  a_1 & a_2 \end{vmatrix}, \quad D_2= \begin{vmatrix} a_0 & a_2 \\  a_1 & a_3 \end{vmatrix},\quad D_3=\begin{vmatrix} a_1 & a_2 \\  a_2 & a_3 \end{vmatrix}. \] The two eigenvalues of $\Lambda$ are \[ \lambda_{1,2}=\frac{D_2 \pm \sqrt{D_2^2 -4D_1D_3}}{2}. \] Note that the cases of $D_1=0$ and $D_3=0$ were already treated in Proposition \ref{onecube}, thus in the following we assume $D_1 \ne 0.$ By computing the eigenvectors of $\Lambda,$ we have 
\begin{equation}
\begin{pmatrix} 0 & -D_3 \\ D_1 & D_2 \end{pmatrix} 
\begin{pmatrix} -\lambda_2 & -\lambda_1 \\   D_1 & D_1 \end{pmatrix}
=\begin{pmatrix} -\lambda_2 & -\lambda_1 \\  D_1 & D_1 \end{pmatrix}
\begin{pmatrix} \lambda_1 & 0 \\ 0 & \lambda_2 \end{pmatrix}
\end{equation}
Now as in the proof of Proposition \ref{cd} take the change of variables 
\begin{equation} \label{cv}
 \begin{pmatrix} x \\ y \end{pmatrix}=\begin{pmatrix} -\lambda_2 & -\lambda_1 \\  D_1& D_1 \end{pmatrix} \begin{pmatrix} u \\ v \end{pmatrix}, 
 \end{equation}
then we have \[ F(x, y)=a u^3 + b v^3 \] for some $a, b \in \C.$ One can determine $a$ and $b$ by comparing the coefficients of the both sides of the preceding equation. As a summary of the previous discussion, we have

\begin{proposition}\label{cc}
Let $F(x,y)=a_0x^3+3a_1x^2y+3a_2xy^2+a_3y^3$ and suppose $D_1 \ne 0$ and $$D_2^2 -4D_1D_3 \ne 0.$$ Let $\lambda_1$ and $\lambda_2$ be the eigenvalues of the matrix $\Lambda$ defined by \eqref{mL}. Then
\begin{equation}\label{completing}
F(x, y)=\frac{\lambda_2a_0-D_1a_1}{\lambda_2 - \lambda_1}(x + \frac{\lambda_1}{D_1} y)^3 + \frac{\lambda_1 a_0 - D_1a_1 }{\lambda_1 - \lambda_2}(x+ \frac{\lambda_2}{D_1} y)^3.
\end{equation}
\end{proposition}

\begin{proof}
Continued with \eqref{cv}, by direct computation we have 
\[ \begin{pmatrix} u \\ v \end{pmatrix} = \frac{1}{(\lambda_1-\lambda_2)D_1} \begin{pmatrix} D_1 & \lambda_1 \\ -D_1 & -\lambda_2 \end{pmatrix} \begin{pmatrix} x \\ y \end{pmatrix}. \] 
Thus we may assume $F(x,y)=\alpha(x + \frac{\lambda_1}{D_1} y)^3 + \beta (x+ \frac{\lambda_2}{D_1} y)^3$ for some $\alpha$ and $\beta.$ 
Then equation \eqref{completing} follows by comparing the coefficients.
\end{proof}

The following is the corresponding result for general cubic equations.

\begin{theorem} \label{rs}
Let $f(x)=a_0x^3+3a_1x^2+3a_2x+a_3$ and keep the conditions for the $a_i$ and the $\lambda_i$ of Proposition \ref{cc}. Then
\begin{equation}\label{cube}
f(x)=\frac{\lambda_2a_0-D_1a_1}{\lambda_2 - \lambda_1}(x + \frac{\lambda_1}{D_1})^3 + \frac{\lambda_1 a_0 - D_1a_1 }{\lambda_1 - \lambda_2}(x+ \frac{\lambda_2}{D_1})^3.
\end{equation}
and the complex roots of $f(x)=0$ are 
\begin{equation} \label{rf}
x_i=\frac{\gamma \omega^i \lambda_1-\lambda_2}{D_1(1-\gamma \omega^i )}, \quad i=0, 1, 2 
\end{equation} 
where $\gamma=\sqrt[3] {\frac {\lambda_2a_0-D_1a_1} {\lambda_1 a_0 - D_1a_1}}$ and $\omega=-\frac{1}{2} +\frac{\sqrt{3}}{2}i.$
\end{theorem}

\begin{proof}
By \eqref{deh} and \eqref{completing}, the cubic equation can be completed cubes as \[ f(x)=\frac{\lambda_2a_0-D_1a_1}{\lambda_2 - \lambda_1}(x + \frac{\lambda_1}{D_1})^3 + \frac{\lambda_1 a_0 - D_1a_1 }{\lambda_1 - \lambda_2}(x+ \frac{\lambda_2}{D_1})^3. \] It is clear that \[ f(x)=0 \Leftrightarrow \frac{\lambda_2a_0-D_1a_1}{\lambda_1 - \lambda_2}(x + \frac{\lambda_1}{D_1})^3 = \frac{\lambda_1 a_0 - D_1a_1 }{\lambda_1 - \lambda_2}(x+ \frac{\lambda_2}{D_1})^3. \] Then one readily obtains the claimed roots for the cubic equation.
\end{proof}

\begin{theorem} \label{nil}
Let $f(x)=a_0x^3+3a_1x^2+3a_2x+a_3$ and suppose $D_1 \ne 0$ and $D_2^2 - 4D_1D_3 = 0.$ Then the roots of $f(x)=0$ are 
\[ x_1=x_2=-\frac{D_2}{2D_1 }, \quad x_3= \frac{D_2}{D_1} - \frac{3a_1}{a_0}. \]
\end{theorem}

\begin{proof}
Let $F(x,y)$ be the homogenization of $f(x).$ Under the present condition, the matrix $\Lambda$ has two equal eigenvalues $\lambda=\lambda_{1,2}=\frac{D_2}{2}.$ It follows that
$Z(F) \cong \C[\epsilon]/(\epsilon^2).$ It is clear that $\Lambda-\lambda I_2 \in Z(F)$ and is square zero. By direct computation, we have \[ \begin{pmatrix} -\lambda & -D_3 \\ D_1 & \lambda \end{pmatrix} \begin{pmatrix} -\lambda & 1-\lambda \\ D_1 & D_1 \end{pmatrix} = \begin{pmatrix} -\lambda & 1-\lambda \\ D_1 & D_1 \end{pmatrix} \begin{pmatrix} 0 & 1 \\ 0 & 0 \end{pmatrix}. \] Let $P$ denote the invertible matrix $\begin{pmatrix} -\lambda & 1-\lambda \\ D_1 & D_1 \end{pmatrix}.$ Take a change of variables $\begin{pmatrix} x \\ y \end{pmatrix}=P \begin{pmatrix} u \\ v \end{pmatrix}$ and denote the resulting binary cubic by $G(u,v).$ Note that $Z(G)=P^{-1} Z(F) P.$ Therefore $\begin{pmatrix} 0 &1 \\ 0 & 0 \end{pmatrix} \in Z(G).$ Then by the condition of center algebras, we can easily observe that  $\frac{\partial^2 G}{\partial u^2}=0.$ It follows that the monomials $u^3$ and $u^2v$ do not appear in $G(u,v).$ Therefore $$F(x,y)=G(u,v)=b_2uv^2+b_3v^3=(b_2u + b_3v)v^2.$$ Note that $v=x+\frac{\lambda}{D_1}y.$ Then by the dehomogenization of $F$ we observe that $(x+\frac{\lambda}{D_1})^2$ is a factor of $f(x).$ In other words, $x=-\frac{D_2}{2D_1 }$ is a double root of $f(x)=0.$ Finally, the third root follows by the Vieta's formula.
\end{proof}

\begin{remark} \label{cf}
\emph{ We apply Theorems \ref{rs} and \ref{nil} to the cubic equation $x^3+px+q=0.$ We assume $pq \ne 0$ to avoid the trivial situation. In this case, $D_1=\frac{p}{3},$ $D_2=q,$ $D_3=-\frac{p^2}{9},$ and $\lambda_{1,2}=\frac{q}{2} \pm \sqrt{\frac{q^2}{4}+\frac{p^3}{27}}.$ In the case of \[ D_2^2 - 4D_1D_3 \ne 0 \Leftrightarrow \frac{p^3}{27}+\frac{q^2}{4} \ne 0, \] by \eqref{rf} we have
\begin{eqnarray*}
x_0&=&\frac{3}{p} \frac{ \lambda_1^{\frac{2}{3}} \lambda_2^{\frac{1}{3}}- \lambda_2}{1-\lambda_1^{-\frac{1}{3}} \lambda_2^{\frac{1}{3}} }=\frac{3}{p}(\lambda_1 \lambda_2)^{\frac{1}{3}} \frac{\lambda_1^{\frac{2}{3}}-\lambda_2^{\frac{2}{3}}} {\lambda_1^{\frac{1}{3}}-\lambda_2^{\frac{1}{3}}} = (-\lambda_2)^{\frac{1}{3}}+(-\lambda_1)^{\frac{1}{3}} \\ &=& \sqrt[3]{-\frac{q}{2} + \sqrt{\frac{q^2}{4}+\frac{p^3}{27}}} + \sqrt[3]{-\frac{q}{2} - \sqrt{\frac{q^2}{4}+\frac{p^3}{27}}}. \end{eqnarray*}
This is exactly the well-known Cardano formula. In case $D_2^2 - 4D_1D_3 = 0 \Leftrightarrow \frac{p^3}{27}+\frac{q^2}{4} = 0,$ then by Theorem \ref{nil} the cubic equation has multiple roots and the roots are \[ x_1=x_2=-\frac{3q}{2p}, \quad x_3=\frac{3q}{p}. \] Note that the third root $x_3$ coincides with  the one obtained from the Cardano formula. Therefore, the Cardano formula can be derived from our approach of completing the cube. Moreover, the terms appearing in the Cardano formula now have some interesting meaning: they are the eigenvalues of a generating matrix of the center algebra of its associated binary cubic. }
\end{remark}

\section{Solving some algebraic equations by completing powers}
The crux of solving cubic equations by completing the cube is that the center algebra of a general binary cubic is nontrivial, namely it contains matrices other than scalar matrices. This approach is easily extended to equations of higher degrees with nontrivial center. In particular, the key structure information of center algebras enables us to complete powers and therefore helps to find radical solutions to some algebraic equations.
 
For convenience, write a complex algebraic equation of degree $d \ge 3$ as
\begin{equation} \label{ed}
f(x)=a_0x^d+{d \choose 1}a_1x^{d-1}+\cdots+{d \choose d-1}a_{d-1}x+a_d.
\end{equation}
We also consider its homogenization
\begin{equation} \label{fd}
F(x,y)=a_0x^d+{d \choose 1}a_1x^{d-1}y+\cdots+{d \choose d-1}a_{d-1}xy^{d-1}+a_dy^d.
\end{equation}

First of all, we compute the center of $F(x,y).$ Suppose $X=(c_{ij})_{2 \times 2} \in Z(F).$ The the condition of center is equivalent to the following system of linear equations
\begin{equation} \label{dlec}
\begin{cases}
a_0c_{12}+a_1(c_{22}-c_{11})-a_2c_{21}=0 \\
a_1c_{12}+a_2(c_{22}-c_{11})-a_3c_{21}=0 \\
 \ \dots \dots \dots \dots \dots \dots \dots \dots \dots \dots \\
a_{d-2}c_{12}+a_{d-1}(c_{22}-c_{11})-a_dc_{21}=0
\end{cases}
\end{equation}
Since the $a_i$ are arbitrary, the rank of the coefficient matrix of the previous system of linear equations is in general $3.$ That is to say, $Z(F) \cong \C$ for a general binary cubic $F.$ See \cite{hlyz2} for an explicit explanation. However, if $Z(F)$ is nontrivial, then $F$ can be completely determined by its center structure. Accordingly, the associated algebraic equation $f(x)=0$ can be solved. In the following, we switch freely between $f$ and $F$ as necessary. 

\begin{theorem} 
Let $f(x)$ and $F(x,y)$ be as \eqref{ed} and \eqref{fd}. 
\begin{itemize}
\item[(1)] $f(x)=(\alpha x+\beta)^d$ if and only if $\frac{a_0}{a_1}=\frac{a_1}{a_2}=\cdots=\frac{a_{d-1}}{a_d}.$
\item[(2)] $f(x)=(\alpha x+\beta)^d+\gamma$ if and only $\frac{a_0}{a_1}=\frac{a_1}{a_2}=\cdots=\frac{a_{d-2}}{a_{d-1}}.$
\item[(3)] $f(x)=\gamma x^d+(\alpha x + \beta)^d$ if and only $\frac{a_1}{a_2}=\cdots=\frac{a_{d-2}}{a_{d-1}}=\frac{a_{d-1}}{a_d}.$
\item[(4)] $f(x)=(\alpha_1 x + \beta_1)^d+(\alpha_2 x + \beta_2)^d$ with $\alpha_1\beta_2-\alpha_2\beta_1 \ne 0$ if and only if $Z(F) \cong \C \times \C.$
\item[(5)] $f(x)=(\alpha_1 x + \beta_1)(\alpha_2 x + \beta_2)^{d-1}$ with $\alpha_1\beta_2-\alpha_2\beta_1 \ne 0$ if and only if $Z(F) \cong \C[\epsilon]/(\epsilon^2).$
\end{itemize}
\end{theorem}

\begin{proof}
The proofs for the corresponding results of cubics can be applied verbatim here.
\end{proof}

Finally, we derive a radical formula for an algebraic equation with nontrivial center. Keep the following notations of Section 3: \[ D_1=\begin{vmatrix} a_0 & a_1 \\  a_1 & a_2 \end{vmatrix}, \quad D_2= \begin{vmatrix} a_0 & a_2 \\  a_1 & a_3 \end{vmatrix},\quad D_3=\begin{vmatrix} a_1 & a_2 \\  a_2 & a_3 \end{vmatrix}, \quad \lambda_{1,2}=\frac{D_2 \pm \sqrt{D_2^2 -4D_1D_3}}{2}. \] Let $\Pi$ denote the following Hankel matrix
\[ 
\begin{pmatrix}
a_0 & a_1 & a_2 \\
a_1 & a_2 & a_3 \\
\vdots & \vdots & \vdots \\
a_{k-2} & a_{k-1} & a_k
\end{pmatrix}.
\]
As the case of $\operatorname{rank} \Pi=1$ is easy and treated in item (1) of the previous theorem, in the rest we focus on the case of $\operatorname{rank} \Pi=2.$ 

\begin{theorem}
Suppose $f(x)=a_0x^d+{d \choose 1}a_1x^{d-1}+\cdots+{d \choose d-1}a_{d-1}x+a_d$ and $a_0 \ne 0.$
\begin{itemize}
\item[(1)] Assume $\operatorname{rank} \Pi=2,$ $D_1 \ne 0$ and $D_2^2 \ne 4D_1D_3.$ Then the roots of $f(x)=0$ are \[ x_i=\frac{\delta \zeta^i \lambda_1 - \lambda_2}{D_1(1-\delta \zeta^i)}, \quad i=0, 1, \dots, d-1 \] where $\delta=\sqrt[d]{\frac{\lambda_2 a_0-D_1a_1}{\lambda_1 a_0-D_1a_1}},$ $\zeta=\cos\frac{2\pi}{d}+i\sin\frac{2\pi}{d}.$
\item[(2)] Assume $\operatorname{rank} \Pi=2,$ $D_1 \ne 0$ and $D_2^2 = 4D_1D_3.$ Then the roots of $f(x)=0$ are \[ x_1=\dots=x_{d-1}=-\frac{D_2}{2D_1}, \ x_d=\frac{(d-1)D_2}{2D_1}-\frac{da_1}{a_0}. \]
\end{itemize}
\end{theorem}

\begin{proof} By the assumptions $\operatorname{rank} \Pi=2$ and $D_1 \ne 0,$ the system of linear equations \eqref{dlec} is determined by the first two rows. Then the center algebra $Z(F)$ is generated by $$\Lambda=\begin{pmatrix} 0 & -D_3 \\ D_1 & D_2 \end{pmatrix}. $$ The eigenvalues of $\Lambda$ are $\lambda_{1,2}=\frac{D_2 \pm \sqrt{D_2^2 -4D_1D_3}}{2}.$
\begin{itemize}
\item[(1)] If $D_2^2 \ne 4D_1D_3,$ then $\lambda_1 \ne \lambda_2$ and so $Z(F) \cong \C \times \C.$ Similar to the proofs of Proposition \ref{cc} and Theorem \ref{rs}, we have 
\[f(x)=\frac{\lambda_2a_0-D_1a_1}{\lambda_2 - \lambda_1}(x + \frac{\lambda_1}{D_1})^d + \frac{\lambda_1 a_0 - D_1a_1 }{\lambda_1 - \lambda_2}(x+ \frac{\lambda_2}{D_1})^d.\] It is clear that \[ f(x)=0 \Leftrightarrow \frac{\lambda_2a_0-D_1a_1}{\lambda_1 - \lambda_2}(x + \frac{\lambda_1}{D_1})^d = \frac{\lambda_1 a_0 - D_1a_1}{\lambda_1 - \lambda_2}(x+ \frac{\lambda_2}{D_1})^d.\] Obviously, $f(x)=0$ has the claimed roots.
\item[(2)] If $D_2^2 = 4D_1D_3,$ then $\lambda_1=\lambda_2=\frac{D_2}{2}$ and thus $Z(F) \cong \C[\epsilon]/(\epsilon^2).$Then similar to the proof of Theorem \ref{nil}, $x=-\frac{\lambda_1}{D_1}=-\frac{D_2}{2D_1}$ is a root of multiplicity $d-1.$ The other root is obtained by the Vieta's formula.
\end{itemize}

\end{proof}

\begin{remark}
\emph{The condition $D_1 \ne 0$ amounts to that of the first two columns of $\Pi$ are linearly independent. One may derive the similar radical formulas when the second column and the third column are linearly independent, or the first column and the third column are linearly independent. We leave the detail to the interested reader.}
\end{remark}

\begin{example} 
\emph{ Consider the quintic equation $31 x^5 + 235 x^4 + 710 x^3 + 1070 x^2 + 805 x + 242=0.$ Then $D_1=-8, \ D_2=-20, \ D_3=-12, \ \lambda_1=-8, \ \lambda_2=-12, \ \delta=\frac{1}{2}.$ So it fits item (1) of the previous theorem and the roots are $x_0=-2$ and $x_n=\frac{3-e^{\frac{2 \pi i}{n}}}{e^{\frac{2 \pi i}{n}} - 2}$ for $1 \le n \le 4.$ }
\end{example}

\begin{example} 
\emph{ Consider the degree $7$ equation $$x^7 - \frac{8}{3} x^6 + \frac{11}{4} x^5 - \frac{5}{4} x^4 + \frac{5}{48} x^3 + \frac{1}{8}x^2 - \frac{3}{64} x + \frac{1}{192}=0.$$ Then $D_1=-\frac{25}{1764}, \ D_2=\frac{25}{1764}, \ D_3=-\frac{25}{7056}, \ \lambda_1=\lambda_2=\frac{25}{3528}.$ So the equation fits item (2) of the previous theorem and the roots are $x_1=\dots=x_6=\frac{1}{2}, \ x_7=-\frac{1}{3}.$ }
\end{example}

\section{Summary}
In this note, we apply nontrivial center algebraic structure to provide radical solutions to some higher degree algebraic equations. In the case of cubic equations, we show that each cubic has a nontrivial center and this enables us to complete the cube, or factorize the equation. For algebraic equations of degree greater than $3,$ we provide very simple and elementary criterion and algorithm to complete powers and obtain radical solutions. The present method only works for very special equations, even for quartic equations. However, if we consider the completion of powers in a broader sense, then we may be able to solve more equations. In the following we take quartic equations as examples to elucidate our idea. 

Let $f(x)=a_0x^4+4a_1x^3+6a_2x^2+4a_3x+a_4$ be a quartic. Then by a suitable change of variable, the quartic can be reduced to $g(y)=y^4+py^2+qy+r.$ In stead of completing $g(y)$ as the sum of two biquadrates, we are content with writing $g(y)$ as a sum of two squares. Of course, this is already enough to solve the quartic equation $g(y)=0.$ We carry out the idea by the method of undetermined coefficients. Suppose \[ y^4+py^2+qy+r=(y^4+2\alpha y^2+\alpha^2)+[(p-2\alpha)y^2+qy+r-\alpha^2] \] and choose $\alpha$ such that the latter term of quadratic is a perfect square. This is equivalent to \[ q^2-4(p-2\alpha)(r-\alpha^2)=0. \] This is a cubic equation in $\alpha$ and clearly such an $\alpha$ is available. Therefore, quartic equations can be solved by a generalized completion of powers. It is of interest whether there is a notion of generalized centers for higher degree forms which governs such a generalization of completing the powers, or even the generalized Waring decompositions of polynomials \cite{fos}. 

Note also that if the binary form $F(x,y)$ has center $Z(F) \cong \C \times \C,$ then the splitting field of its dehomogenization $f(x)$ is usually a Kummer extension \cite{l} over a suitable ground filed. It is also of interest whether all Kummer extensions appear in this way.

%

\end{document}